\documentclass[12pt]{amsart}
\usepackage{amssymb}

\theoremstyle{plain}
\newtheorem{thm}{Theorem}[section]

\newtheorem{corollary}[thm]{Corollary}

\theoremstyle{definition}
\newtheorem{remark}[thm]{Remark}

\numberwithin{equation}{section}

\def\Rn{\mathbb{R}^n}

\def\sup{\operatornamewithlimits{sup}}

\def\esssup{\operatornamewithlimits{ess\,sup}}
\def\essinf{\operatornamewithlimits{ess\,inf}}

\def\P{\mathcal P}
\def\B{\mathcal B}

\newtoks\by
\newtoks\paper
\newtoks\book
\newtoks\jour
\newtoks\yr
\newtoks\pages
\newtoks\vol
\newtoks\publ
\def\ota{{\hbox\vol{???}}}
\def\cLear{\by=\ota\paper=\ota\book=\ota\jour=\ota\yr=\ota
\pages=\ota\vol=\ota\publ=\ota}
\def\endpaper{\the\by, \textit{\the\paper},
{\the\jour} \textbf{\the\vol} (\the\yr), \the\pages.\cLear}
\def\endbook{\the\by, \textit{\the\book}, \the\publ.\cLear}
\def\endprep{\the\by, \textit{\the\paper}, \the\jour.\cLear}

\title[Extensions of Rubio de Francia's
extrapolation theorem in]{Extensions of Rubio de Francia's
extrapolation theorem in variable Lebesgue space and application}

\author{Amiran Gogatishvili}
\address{Amiran Gogatishvili \\
Institute of Mathematics of the \\
Academy of Sciences of the Czech Republic \\
\'Zitna 25 \\
115 67 Prague 1, Czech Republic}
\email{gogatish@math.cas.cz}

\author{Tengiz Kopaliani}
\address{Tengiz Kopaliani \\
Department of Mechanics and Mathematics \\
Tbilisi State University \\
Chavchavadze av. 1 \\
0128, Tbilisi, Georgia}
\email{tengiz.kopaliani@tsu.ge}


\thanks{The research of the  authors was  supported by Shota Rustaveli National Science Foundation
grants no. 31/48 (Operators in some function spaces and their
applications in Fourier Analysis) and  no. DI/9/5-100/13 (Function spaces, weighted inequalities for integral operators and problems of summability of Fourier series). The research of the  second author was  partly supported by grant no. P201-13-14743S  of the
Grant Agency of the Czech Republic and RVO: 67985840.}

\begin{document}

\begin{abstract}
We obtain one variant of the  extrapolation theorem of Rubio de Fracia for variable exponent Lebesgue spaces.
As a consequence we obtain conditions guarantee boundedness of strongly singular integral operators,
singular integral operators with rough kernels, fractional maximal operators related to spherical means, Bochner-Riesz
operators
in variable Lebesgue spaces.

\end{abstract}

\maketitle

\section{Introduction}

Given a measurable
function $p :\mathbb{R}^{n}\longrightarrow
[1,\infty),\,L^{p(\cdot)}(\mathbb{R}^{n})$ denotes the set of
measurable functions $f$ on $\mathbb{R}^{n}$ such that for some
$\lambda>0$
$$
\int_{\mathbb{R}^{n}}\left(\frac{|f(x)|}{\lambda}
\right)^{p(x)}dx<\infty.
$$
This set becomes  a Banach function space  when equipped with the
norm
$$
\|f\|_{p(\cdot)}=\inf\left\{\lambda>0:\,\,\int_{\mathbb{}}\,
\left(\frac{|f(x)|}{\lambda}\right)^{p(x)}dx\leq 1 \right\}.
$$

The Lebesgue spaces $L^{p(\cdot)}(\mathbb{R}^{n})$ with variable
exponent  and the corresponding variable Sobolev spaces
$W^{k,p(\cdot)}(\mathbb{R}^{n})$ are of interest for their
applications to modeling problems in physics, and to the study of
variational integrals and partial differential equations with
non-standard growth condition (see \cite{DHHR}, \cite{CUF}).The
space $L^{p(\cdot)}(\mathbb{R}^{n})$ have many properties in common
with the standard $L^{p}$ spaces. For use below we highlight the
fact that the associate space of $L^{p(\cdot)}(\mathbb{R}^{n})$ is
$L^{p'(\cdot)}(\mathbb{R}^{n}),$ where the conjugate exponent
function $p'(\cdot)$ is defined by $1/p(x)+1/p'(x)=1$ with
$1/\infty=0.$

Define $\mathcal{P}^{0}$  to be the set of measurable functions $p:
\Rn \to(0,\infty)$ such that
 $$
 p^-=\essinf_{x\in \Rn}p(x)>0,\,\,\,\,p^+=\esssup_{x\in
 \Rn}p(x)<\infty.
 $$

Given $p(\cdot)\in\mathcal{P}^{0},$ we can define the space
$L^{p(\cdot)}(\mathbb{R}^{n})$ as above.  This is equivalent to
defining it to be the set of all functions $f$ such that
$|f|^{p_{0}}\in L^{q(\cdot)}(\mathbb{R}^{n}), $ where
$0<p_{0}<p^{-}$ and $q(\cdot)=p(\cdot)/p_{0}.$ We can define a
quasi-norm on this space by
$$
\|f\|_{p(\cdot)}=\||f|^{p_{0}}\|_{q(\cdot)}^{1/p_{0}}.
$$

By a weight we mean a non-negative, locally integrable function $w$.
Given a weight $w$, $L^p(w)$
will denote the weighted Lebesgue space with norm
$$\|f\|_{p,w}:=\left(\int_{\Rn}|f(x)|^pw(x)dx\right)^{\frac{1}{p}}.$$

Central to the study of weights are the so-called $A_{p}$ weights,
we say $w\in A_{p}$ if there exists
a constant $C$ such that for every cube $Q\subset \Rn$,
\begin{align*}
\frac{1}{|Q|}\int_{Q}w(x)dx\left(\frac{1}{|Q|}\int_{Q}w(x)^{1-p'}dx\right)^{p-1}\leq
C<\infty, \quad &\text{if}\quad  1<p<\infty\\
\intertext{and}
\frac{1}{|Q|}\int_Qw(x)dx\leq C\essinf_{x\in Q}w(x), \quad &\text{if} \quad p=1.
\end{align*}
Hereafter  $\mathcal{F}$ will denote a family of pairs  $(f,g)$ of
non-negative, measurable functions on $\Rn$.  We say that an
inequality
$$
\int_{\mathbb{R}^{n}}f(x)^{p_{0}}w(x)dx\leq
C\int_{\mathbb{R}^{n}}g(x)^{p_{0}}w(x)dx
$$
holds for any $(f,g)\in\mathcal{F}$ and $w\in A_{q}$ (for some
$q,\,\,1\leq q<\infty$), we mean that it holds for any pair in
$\mathcal{F}$ such that the left-hand side is finite, and the
constant $C$ depends only $p_{0}$ and the $A_{q}$ constant of $w.$

Let $B(x,r)$ denote the open ball in $\mathbb{R}^{n}$ of radius $r$
and center $x.$ By $|B(x,r)|$ we denote the $n-$dimensional Lebesgue
measure of $B(x,r).$  The Hardy-Littlewood maximal operator $M$ is
defined on locally integrable functions $f$ on $\mathbb{R}^{n}$ by
the formula
$$
Mf(x)=\sup_{r>0}\frac{1}{|B(x,r)|}\int_{B(x,r)}|f(y)|dy.
$$

Let $\mathcal{B}(\mathbb{R}^{n})$ be the set of exponents such that
Hardy-Littlewood maxcimal operator is bounded on
$L^{p(\cdot)}(\mathbb{R}^{n}).$ (For  information on exponents from
$\mathcal{B}(\mathbb{R}^{n})$ see the  monographs \cite{CUF, DHHR}).

In \cite{CUFMP} the authors  extending the classical extrapolation
method of Rubio de Francia (\cite{RF1},\cite{RF2},\cite{RF3}) for
variable exponent Lebesgue spaces and  showed that many classical
operators in harmonic analysis such as singular integrals,
commutators and fractional integrals are bounded on the variable
Lebesgue spaces $L^{p(\cdot)}$ whenever the Hardy-Littlewood maximal
operator is bounded on $L^{p(\cdot)}.$

\begin{thm} $($\cite{CUMP},\cite{CUFMP}$)$ \label{thm1.1} Given a family $\mathcal{F},$ suppose that for some
  $p_0 $,  $0< p_0 <\infty,$ and for every weight  $w\in A_{1}$,
$$
\int_{\Rn} f(x)^{p_0} w(x)dx\le C_0 \int_{\Rn} g(x)^{p_0}
w(x)dx,\,\,\,\,\,(f,g)\in\mathcal{F},
$$
where $C_0$ depends only $p_0$ and the $A_1$ constant of $w$. 
Let $p(\cdot)\in \mathcal{P}^{0}$ be such that $p_0<p_{-}$, and  $(p(\cdot)/p_{0})'\in \mathcal{B}(\mathbb{R}^{n}).$
 Then for all
$(f,g)\in\mathcal{F}$ such that $f\in L^{p(\cdot)}(\mathbb{R}^{n})$,
$$\|f\|_{p(\cdot)}\le C\|g\|_{p(\cdot)},$$
where the constant $C$ is independent of the pair $(f,g)$.
\end{thm}

 "Of-diagonal" generalization of Theorem 1.1 is following theorem (In
 the classical setting the extrapolation theorem of Rubio de Francia
 was extended in this manner by  Harboure, Mac\'{\i}as and Segovia
 \cite{HMS}).

\begin{thm} $($\cite{CUMP},\cite{CUFMP}$)$
 Given a family $\mathcal{F}$, assume that for some $p_{0}$ and $q_{0}$  , $0<p_{0}\leq q_{0}<\infty,$
 and every weight $w\in A_{1},$
 $$
 \left(\int_{\mathbb{R}^{n}}f(x)^{q_{0}}w(x)dx\right)^{1/q_{0}}\leq
 C_0\left(\int_{\mathbb{R}^{n}}g(x)^{p_{0}}w(x)^{ p_{0}/q_{0}}dx\right)^{1/p_{0}},\,\,\,\,\,(f,g)\in\mathcal{F}.
 $$
Given $p(\cdot)\in \mathcal{P}^{0}$ such that $p_{0}<p_{-}\leq
p_{+}<\frac{p_{0}q_{0}}{q_{0}- p_{0}},$ define the function
$q(\cdot)$ by
$$
\frac{1}{p(x)}-\frac{1}{q(x)}=\frac{1}{p_{0}}-\frac{1}{q_{0}},\,\,\,x\in\mathbb{R}^{n}.
$$
If $(q(\cdot)/q_{0})'\in \mathcal{B}(\mathbb{R}^{n}),$ then for all
$(f,g)\in\mathcal{F}$ such that $f\in L^{p(\cdot)}(\mathbb{R}^{n}),$
$$
\|f\|_{q(\cdot)}\leq C\|g\|_{p(\cdot)}.
$$
\end{thm}

By Johnson and Neugebauer \cite{JN} and by Duoandikoetxea et al.
\cite{DMOS} have been obtained a restricted range extrapolation
theorem by restricting the class of weights. We will prove analogous
kind theorem for variable exponent Lebesgue spaces.

We prove following theorems.

\begin{thm} \label{thm1.3}
Given a family $\mathcal{F},$ suppose that for some
  $p_0,\,\delta, $  $0< p_0 <\infty,\,0 < \delta <
1,$ and for every weight  $w\in A_{1}$
$$
\int_{\Rn} f(x)^{p_0} w^{\delta}(x)dx\le C_0 \int_{\Rn} g(x)^{p_0}
w^{\delta}(x)dx,\,\,\,\,\,(f,g)\in\mathcal{F}.
$$
Let  $\delta(p(\cdot)/p_{0})'\in \mathcal{B}(\mathbb{R}^{n})$,
$p_0<p^{-}\le p^+<\frac{p_{0}}{1-\delta}.$ Then for all
$(f,g)\in\mathcal{F}$ such that $f\in L^{p(\cdot)}(\mathbb{R}^{n})$
$$\|f\|_{p(\cdot)}\le C\|g\|_{p(\cdot)}.$$
\end{thm}

\begin{thm} \label{thm1.4}
 Given a family $\mathcal{F}$, assume that for some $p_{0},$  $q_{0}$
 and $\delta$
 , $0<p_{0}\leq q_{0}<\infty,$  $0<\delta<1$  and every weight $w\in A_{1},$
 \begin{equation} \label{1.4}
 \left(\int_{\mathbb{R}^{n}}f(x)^{q_{0}}w(x)^{\delta}dx\right)^{1/q_{0}}\leq
 C_0\left(\int_{\mathbb{R}^{n}}g(x)^{p_{0}}w(x)^{\delta p_{0}/q_{0}}dx\right)^{1/p_{0}},\,\,\,\,\,(f,g)\in\mathcal{F}.
 \end{equation}
Given $p(\cdot)\in \mathcal{P}^{0}$ such that $p_{0}<p^{-}\leq
p_{+}<\frac{p_{0}q_{0}}{q_{0}-\delta p_{0}},$ define the function
$q(\cdot)$ by
\begin{equation}\label{1.1}
\frac{1}{p(x)}-\frac{1}{q(x)}=\frac{1}{p_{0}}-\frac{1}{q_{0}},\,\,\,x\in\mathbb{R}^{n}.
\end{equation}
If $\delta(q(\cdot)/q_{0})'\in \mathcal{B}(\mathbb{R}^{n}),$ then
for all $(f,g)\in\mathcal{F}$ such that $f\in
L^{p(\cdot)}(\mathbb{R}^{n}),$
$$
\|f\|_{q(\cdot)}\leq C\|g\|_{p(\cdot)}.
$$
\end{thm}

Note that Theorem \ref{thm1.3} is a particular case of Theorem \ref{thm1.4} with
$p_{0}=q_{0}.$

The following theorem in case $p_{0}=2$ proved by Duoandikoetxea,
\emph{et. al}. in \cite{DMOS}. The case $p_{0}\neq2$  see
\cite{CUMP}.
\begin{thm}\label{thm1.5}
Given $\delta,\,\,0<\delta<1,$  suppose that for all $w\in
A_{p_{0}},\,\,1<p_{0}<\infty$
$$
\int_{\mathbb{R}^{n}}f(x)^{p_{0}}w^{\delta}(x)dx\leq
C\int_{\mathbb{R}^{n}}g(x)^{p_{0}}w^{\delta}(x)dx,\,\,\,\,(f,g)\in
\mathcal{F}.
$$
Then for all
$p,\,\,\frac{p_{0}}{1+\delta(p_{0}-1)}<p<\frac{p_{0}}{1-\delta},$
and every $w^{\frac{p_{0}}{p_{0}-p(1-\delta)}}\in
A_{\frac{p_{0}p\delta}{p_{0}-p(1-\delta)}}$
$$
\int_{\mathbb{R}^{n}}f(x)^{p}w(x)dx\leq
C\int_{\mathbb{R}^{n}}g(x)^{p}w(x)dx.
$$
\end{thm}

Using the fact that $A_{1}\subset A_{p}$ from Theorem \ref{thm1.3} and
Theorem \ref{thm1.5}, we obtain

\begin{corollary} \label{cor1.6}
Given $\delta,\,\,0<\delta<1,$  suppose that for all $w\in
A_{p_{0}},\,\,1<p_{0}<\infty$
$$
\int_{\mathbb{R}^{n}}f(x)^{p_{0}}w^{\delta}(x)dx\leq
C_0\int_{\mathbb{R}^{n}}g(x)^{p_{0}}w^{\delta}(x)dx,\,\,\,\,(f,g)\in
\mathcal{F}.
$$
Let $\delta_{\ast}(p(\cdot)/p_{\ast})'\in\mathbb{B}(\mathbb{R}^{n})$
where $\delta_{\ast}=\frac{p_{0}-p_{\ast}(1-\delta)}{p_{0}}$ and
$$
\frac{p_{0}}{1+\delta(p_{0}-1)}<p_{\ast}<p^{-}\leq
p^{+}<\frac{p_{0}}{1-\delta}.
$$
Then
for all $(f,g)\in\mathcal{F}$ such that $f\in
L^{p(\cdot)}(\mathbb{R}^{n})$
$$
\|f\|_{p(\cdot)}\leq C\|g\|_{p(\cdot)}.
$$

\end{corollary}

\begin{remark} We can restate the hypotheses of Corollary \ref{cor1.6} in term $w^{\delta}\in A_q\cap RH_{\frac{1}{\delta}}$, where $q =\frac{p-1+\delta}{\delta}$.
 and $RH_s$ is  class of weights satisfying  the reverse H\"older inequality: given $s$, $1 < s < \infty$, we say that
$w\in RH_{s}$ if for every cube $Q\subset \Rn$
$$ \left(\frac{1}{|Q|}\int_Q w(x)^s dx\right)^{\frac{1}{s}}\le \frac{C}{|Q|}\int_Q w(x) dx.
$$ 
To do so, we need to use the following equivalence: for all
$\delta$, $p$, $0<\delta <1$, $1 < p < \infty$,  $w\in A_p$ if and only if $w^{\delta}\in A_q\cap RH_{\frac{1}{\delta}}$, this was proved by Johnson and Neugebauer \cite{JN}.
 
\end{remark}

\section{Proof of Theorem \ref{thm1.4}}

First note that from \eqref{1.1} obtain that
$$
\frac{1}{p^{-}}-\frac{1}{q^{-}}=\frac{1}{p^{+}}-\frac{1}{q^{+}}=\frac{1}{p_{0}}-\frac{1}{q_{0}}.
$$
From $p_{0}<p^{-}\leq p^{+}<\frac{p_{0}q_{0}}{q_{0}-\delta p_{0}},$
we obtain  $q_{0}<q^{-}\leq q^{+}<\frac{q_{0}}{1-\delta}.$

Let $X=L^{p(\cdot)/p_{0}}(\mathbb{R}^{n})$ and
$Y=L^{q(\cdot)/q_{0}}(\mathbb{R}^{n}).$

 Let $q^{-}<q^{+},$ (for the case  $q^{-}=q^{+}$ see \cite{CUMP}) then we can rewrite the estimate
$q^{+}<\frac{q_{0}}{1-\delta}$ in the form
$1<\frac{\delta q^{+}}{q^{+}-q_{0}}.$

We have
$$
1<\frac{\delta q^{+}}{q^{+}-q_{0}}\leq
\frac{\delta q(x)}{q(x)-q_{0}}=\delta\left(\frac{q(x)}{q_{0}}\right)',\,\,\,x\in\mathbb{R}^{n}.
$$
Define the following operator
$\mathcal{M}h(x)=M(h^{1/\delta})^{\delta},$ where $M$ is
Hardy-Littlewood maximal operator.  Note that operator $\mathcal{M}$
is bounded on $Y'.$  Indeed
$$
\|\mathcal{M}\|_{Y'}=\|(M(h^{1/\delta}))^{\delta}\|_{Y'}
=\|M(h^{1/\delta})\|_{\delta(q(\cdot)/q_{0})'}^{\delta}
$$
$$
\leq C\|h^{1/\delta}\|_{\delta(q(\cdot)/q_{0})'}^{\delta} \leq
C\|h\|_{Y'}.
$$

Since the  operator $\mathcal{M}$ is bounded on $Y',$ we can define
the Rubio de Francia iteration algorithm:
$$
\mathcal{R}h(x)=\sum_{k=0}^{\infty}\frac{\mathcal{M}^{k}h(x)}{2^{k}\|\mathcal{M}\|_{Y'}^{k}},
$$
where, for $k\geq1,$ $\mathcal{M}^{k}=\mathcal{M}\circ
\mathcal{M}\cdot\cdot\cdot\circ \mathcal{M}$ denotes $k$ iterations
of the operator $\mathcal{M}$, and  $\mathcal{M}^{0}h=|h|$.

It is follows immediately from the definition that $h(x)\leq\mathcal{R}h(x)$
and $\|\mathcal{R}\|_{Y'}\leq 2\|h\|_{Y'}.$

By the sublinearity   of $\mathcal{M}$
$$
(M(\mathcal{R}h)^{1/\delta}))^{\delta}=\mathcal{M}(\mathcal{R}h)\leq\sum_{j=0}^{\infty}\frac{\mathcal{M}^{j+1}}{2^{k}\|\mathcal{M}^{j}\|_{X'}}\leq
C\cdot\mathcal{R}(h).
$$
Therefore, $(\mathcal{R}h)^{1/\delta}\in A_{1}.$

We can now prove the desired inequality. Fix $(f,g)\in\mathcal{F}.$
Since $Y$ is a Banach function space
$$
\|f\|_{q(\cdot)}^{q_{0}}=\|f^{q_{0}}\|_{Y}=\sup\left\{\int_{\mathbb{R}^{n}}|f(x)|^{q_{0}}h(x)dx:\,\,\,h\in
Y',\,\,\|h\|_{Y'}\leq1\right\}.
$$
Since $f$ is non-negative, we may also restrict the supremum to
non-negative $h.$ Therefore, it will suffice to fix a function $h$
and show that
$$
\int_{\mathbb{R}^{n}}f(x)^{q_{0}}h(x)dx\leq
C\|g\|_{p(\cdot)}^{q_{0}}
$$
with a constant independent of $h.$

We have
\begin{align*}
\int_{\mathbb{R}^{n}}f(x)^{q_{0}}h(x)dx&\leq\int_{\mathbb{R}^{n}}f(x)^{p_{0}}\mathcal{R}h(x)dx\\
&\leq\|f^{q_{0}}\|_{Y}\cdot\|\mathcal{R}h\|_{Y'}\\
&\leq
C\|f\|_{q(\cdot)}^{q_{0}}\|h\|_{Y'}<\infty,
\end{align*}
where we have used that $h \le\mathcal{R}h$. Since  $(\mathcal{R}h)^{1/\delta}\in A_{1}$, by using our hypothesis \eqref{1.4} we obtain

\begin{align*}
\int_{\mathbb{R}^{n}}f(x)^{q_{0}}h(x)dx&\leq\int_{\mathbb{R}^{n}}f(x)^{q_{0}}\mathcal{R}h(x)dx\\
&\le 
C\left(\int_{\mathbb{R}^{n}}g(x)^{p_{0}}\mathcal{R}h(x)^{p_{0}/q_{0}}dx\right)^{q_{0}/p_{0}}\\
&\leq
C\|g^{p_{0}}\|_{X}^{q_{0}/p_{0}}\|(\mathcal{R}h)^{p_{0}/q_{0}}\|_{X'}^{{q_{0}/p_{0}}}\\
&=\|g\|_{p(\cdot)}^{q_{0}}\|(\mathcal{R}h)^{q_{0}/p_{0}}\|_{X'}^{{q_{0}/p_{0}}}.
\end{align*}

Note that
$$
\|(\mathcal{R}h)^{q_{0}/p_{0}}\|_{X'}^{{q_{0}/p_{0}}}=\|\mathcal{R}h)\|_{Y'}\leq
C\|h\|_{Y'}=C.
$$
$\Box$

\section{Applications}

In this section we give a number of applications of Theorems~\ref{thm1.3} and \ref{thm1.4} and 
 Corollary \ref{cor1.6}, to show that a wide variety of classical operators
are bounded on the variable $L^{p(\cdot)}$ spaces.
 
 We start to  introduce the most important condition on the exponent
in the study of variable exponent spaces, the log-H\"older continuity condition.

We say that a function $p: \Rn \to(0,\infty)$ is locally
log-H\"older continuous on $\Rn$ if there exists $c_1 > 0$ such that
$$|p(x)-p(y)|\le c_1 \frac{1}{\log(e + 1/|x- y|)}$$
for all $x, y\in \Rn$.  We say that $p(\cdot)$ satisfies the
log-H\"older decay condition if there exist $p_\infty\in (0,\infty)$
and a constant $c_2 > 0$ such that
$$|p(x)-p_\infty|\le c_2 \frac{1}{\log(e + |x|)}$$
for all $x\in \Rn$. We say that $p(\cdot)$ is globally log-H\"older
continuous in $\Rn$  if it is locally log-H\"older continuous and
satisfies the log-H\"older decay condition.

If $p: \mathbb{R}^{n}\to (1,\infty)$ is globally log-H\"older
continuous function in $\Rn$ and $p^->1$, then the classical
boundedness theorem for the Hardy-Littlewood maximal operator can be
extended to $L^{p(\cdot)}(\mathbb{R}^{n})$ (see \cite{CUF, DHHR}).
This class of exponent we denote  by  $\P_{\log}.$ For the class of
exponents $\P_{\log}$ we have

\begin{corollary} \label{cor3.1}
Given $\delta,\,\,0<\delta<1,$  suppose that for all $w_{0}\in
A_{p_{0}},\,\,1<p_{0}<\infty$
$$
\int_{\mathbb{R}^{n}}f(x)^{p_{0}}w_{0}^{\delta}(x)dx\leq
C_0\int_{\mathbb{R}^{n}}g(x)^{p_{0}}w_{0}^{\delta}(x)dx,\,\,\,\,(f,g)\in
\mathcal{F}.
$$
Let $p(\cdot)\in \P_{\log}$ and
$$
\frac{p_{0}}{1+\delta(p_{0}-1)}<p^{-}\leq
p^{+}<\frac{p_{0}}{1-\delta}.
$$
Then for all $(f,g)\in\mathcal{F}$ such that $f\in
L^{p(\cdot)}(\mathbb{R}^{n})$
$$
\|f\|_{p(\cdot)}\leq C\|g\|_{p(\cdot)}.
$$
\end{corollary}
\begin{proof} Fix $p_{\ast}$ and positive small $\varepsilon$ such that
$\frac{p_{0}}{1+\delta(p_{0}-1)}<p_{\ast}<p^{-}$ and
$p^{+}<p_{0}/(1-\delta)-\varepsilon.$ We have

$$
\frac{p_{0}-p_{\ast}(1-\delta)}{p_{0}}(p(\cdot)/p_{\ast})'\geq\frac{p_{0}-p_{\ast}(1-\delta)}{p_{0}}\left(\frac{p_{0}/(1-\delta)-\varepsilon)}{p_{\ast}}\right)'
$$
$$
=\frac{p_{0}-p_{\ast}(1-\delta)}{p_{0}}\cdot\frac{p_{0}-\varepsilon(1-\delta)}{p_{0}-\varepsilon(1-\delta)-(1-\delta)p_{\ast}}>1.
$$
It is not hard to proof
$\frac{p_{0}-p_{\ast}(1-\delta)}{p_{0}}(p(\cdot)/p_{\ast})'\in\P_{\log}$.
From Corollary 1.6 we obtain desired result.
\end{proof}

\

\textbf{Singular integrals with rough  kernels}. 

We obtain boundedness of the  singular integral operator with "rough" kernel in variable exponent Lebesgue space we need the weighted inequalities. 
 \begin{thm}\cite{W} \label{thm3.2} Let $n\geq2$, $1<r<\infty$ and  let $Tf(x)=p.v.K\ast f(x)$
be singular integral operator with "rough"  kernel
$$
K(x)=h(|x|)\frac{\Omega(x)}{|x|^{n}},
$$
where $\Omega$ is homogeneous of degree $0$ on $\mathbb{R}^{n}$,
$\Omega\in L^{r}(S^{n-1}),$ where  $S^{n-1}$ denote the unit sphere in $\mathbb{R}^{n}$.  $\Omega$ has average $0$
on $ S^{n-1},$ and $h$ is a measurable function on $(0,\infty)$
satisfying
$$
\int_{R}^{2R}|h(t)|^{r}dt\leq CR\,\,\,\,\,\,\,\mbox{for
all}\,\,\,R>0.
$$
Then $T$ is bounded on
$L^{p}(w)(\mathbb{R}^{n})$,
$$ \|Tf\|_{p,w}\le C\|f\|_{p,w},$$
 in each of the following situations: 

(A) if $r'\leq p<\infty,$ and $w\in A_{\frac{p}{r'}},$ or 

(B) if  $1<p\le r, p\not=\infty$ and $w^{\frac{-1}{p-1}}\in A_{\frac{p'}{r'}}$, or  

(C) if $1<p<\infty$ and $w^{r'}\in A_{p}.$ 
\end{thm}

for power weight was consider in \cite{Du}.

\

The following result concerning singular integral operator with "rough" kernel is known.
\begin{corollary} $($\cite{CUFMP}$)$
Let $p(\cdot)\in \mathcal{P}_{\log}$, $1<r<\infty$ and $r'<p^{-}.$ Then singular integral operator with "rough" kernel
$T$ is bounded on $L^{p(\cdot)}(\mathbb{R}^{n}).$
\end{corollary}
\begin{proof} Let $r'<p_{0}<p^{-}$.  As $A_1\subset A_{\frac{p_0}{r'}}$, by Theorem \ref{thm3.2}(A)
$$ \|Tf\|_{p_0,w}\le C\|f\|_{p_0,w},\quad \text{for every}\quad   w\in A_1.$$
 Using Theorem \ref{thm1.1} for $(|Tf|,|f|)$, we obtain that 
  $$ \|Tf\|_{p(\cdot)}\le C\|f\|_{p(\cdot)}.$$
\end{proof}

\

\begin{corollary} 
Let $p(\cdot)\in \mathcal{P}_{\log}$, $1<r<\infty$ and $p^{+}<r.$ Then  singular integral operator with "rough" kernel
$T$ is bounded on $L^{p(\cdot)}(\mathbb{R}^{n}).$
\end{corollary}
\begin{proof}
Let $1<p_0<p^{-}$.   By Theorem \ref{thm3.2}(B)
$$ \|Tf\|_{p_0,w}\le C\|f\|_{p_0,w},$$
for every   $w^{-\frac{1}{p_0-1}}\in A_{\frac{p_0'}{r'}}$. As we know  $w^{-\frac{1}{p_0-1}}\in A_{\frac{p_0'}{r'}}$ if and only if  $w^{\frac{r}{r-p_0}}\in A_{\frac{p_0{r-1}}{r'}}$. Using $A_1\subset A_{\frac{p_0{r-1}}{r'}}$ we have
 $$
\int_{\mathbb{R}^{n}}Tf(x)^{p_{0}}w_{0}^{\delta}(x)dx\leq
C_0\int_{\mathbb{R}^{n}}g(x)^{p_{0}}w_{0}^{\delta}(x)dx, \quad \text{for every} \quad  w_0\in A_1,
$$
 where $\delta=\frac{r-p_0}{r}$. 
 
 It is easy to see that $\frac{p}{1-\delta}=r$ and $ \delta(p^{+}/p_0)^{'}>1$ if $p^{+}<r$. Therefore, $\delta(p(\cdot)/p_0)^{'}\in \B(\Rn)$  and using Theorem \ref{thm1.3} for $(|Tf|,|f|)$, we obtain that 
  $$ \|Tf\|_{p(\cdot)}\le C\|f\|_{p(\cdot)}.$$
 \end{proof}

\begin{corollary}
Let $p(\cdot)\in \mathcal{P}_{\log}$, $1<r<\infty$  and for some  $p_0$, $1<p_{0}<\infty$, 
$\frac{p_{0}r}{r+(r-1)(p_{0}-1)}<p^{-}\leq p^{+}<rp_{0}.$ Then  singular integral operator with "rough" kernel
 $T$ is bounded on $L^{p(\cdot)}(\mathbb{R}^{n}).$
\end{corollary}
\begin{proof} Let $1<p_{0}<\infty$.  by Theorem \ref{thm3.2}(C)
$$ \int_{\Rn}|Tf(x)|^{p_0}w(x)^{\delta}dx\le C\int_{\Rn}|f(x)|w(x)^{\delta}(x)dx,\quad \text{for every}\quad   w\in A_{p_0},$$
 where $\delta=\frac{1}{r-1}$.
 Using Corollary \ref{cor3.1} for $(|Tf|,|f|)$, we get 
  $$ \|Tf\|_{p(\cdot)}\le C\|f\|_{p(\cdot)}.$$
\end{proof}

\
 
\textbf{Strongly singular integrals}. Let $\theta(\xi)$ be a smooth
radial cut-off function $\theta(\xi)=1$ if $|\xi|\geq1$ and
$\theta(\xi)=0$ if $|\xi|\leq1/2.$ We will consider the multipliers
$$
\widehat{T_{b,a}f}(\xi)=\theta(\xi)\frac{e^{i|\xi|^{b}}}{|\xi|^{a}}\widehat{f}(\xi),
$$
where $0<b<1$ and $0<a<nb/2.$ Fefferman \cite{Fe} proved that if
$0<a<nb/2,$ then for $p,$ such that $|1/p-1/2|\leq a/nb,$ then
$$
\|T_{b,a}\|_{p}\leq c_{p}\|f\|_{p}.
$$

The weighted extension of Fefferman's theorem was obtained by
Chanillo \cite{Ch}. Indeed If $\alpha=nb|1/p-1/2|,$ and $w\in
A_{p},$ then for $1<p<\infty,$ $\alpha\leq a\leq nb/2,$ and for
$\gamma,$ such that $\gamma=(a-\alpha)/(nb/2-\alpha)$ we have
$$
\|T_{b,a}f\|_{p,w^{\gamma}}\leq C_{p}\|f\|_{p,w^{\gamma}}.
$$

Using the Corollary \ref{cor3.1} we obtain
\begin{corollary}
Let $p(\cdot)\in \mathcal{P}_{\log}$. Let for some $1<p_0<\infty$, $\alpha=nb|1/p_0-1/2|,$ $\alpha\leq a\leq nb/2,$ and $\gamma=(a-\alpha)/(nb/2-\alpha)$,   
$\frac{p_{0}}{1+\gamma(p_{0}-1)}<p^{-}\leq
p^{+}<\frac{p_{0}}{1-\gamma}$. Then operator $T_{b,a}$ is bounded on
$L^{p(\cdot)}(\mathbb{R}^{n}).$
\end{corollary}

\

\textbf{Fractional maximal functions related to spherical means.}
Denote by $\mu_{t}$ the normalized surface measure on the sphere in
$\mathbb{R}^{n}$ with center $0$ and radius $t.$ The maximal
operator related to spherical means is given by
$$
\mathcal{M}^{\alpha}=\sup\limits_{t>0}|t^{\alpha}\mu_{t}\ast f|.
$$
In  paper \cite{CGCG} the authors investigate weighted
$L^{p}\rightarrow L^{q}$ estimate for the maximal operators
$\mathcal{M}^{\alpha}.$

\begin{thm} \label{thm3.5}$($\cite{CGCG}$)$ Suppose that $n/n-1<p<q<n,\,\,n>2,$
that $\alpha=n/p-n/q,$ and that $\max\{0,1-q/p'\}<\gamma\leq1-q/n.$
Suppose also that $w$ is in $A_{s},$ where
$$
s=\frac{q+2p'\gamma-p'}{p'\gamma}.
$$
Then there exists a constant $C$ such that
$$
\|\mathcal{M}^{\alpha}f\|_{q,w^{\gamma}}\leq C
\|f\|_{p,w^{p\gamma/q}}.
$$
\end{thm}

From Theorem \ref{thm1.4} and Theorem \ref{thm3.5} we obtain
\begin{corollary}
Let $0<\alpha<n-2,\,\,n>2,$ $p(\cdot)\in \mathcal{P}_{\log} $ and
$n/(n-1)<p^{-}\leq p^{+}<n/(1+\alpha).$ Define $q(\cdot)$ by
$$
\frac{1}{p(x)}-\frac{1}{q(x)}=\frac{\alpha}{n},\,\,\,x\in\mathbb{R}^{n}.
$$
Then
$$
\|\mathcal{M}^{\alpha}f\|_{q(\cdot)}\leq C\|f\|_{p(\cdot)}.
$$
\end{corollary}

\begin{proof} Fix $ p_{\ast}$ such that $\frac{n}{n-1}<p_{\ast}<p^{-}$
 and define $q_{\ast}$ from equation $1/p_{\ast}-1/q_{\ast}=\alpha/n.$
Note that $n/(n-1)<p_{\ast}<q_{\ast}<n.$ From Theorem \ref{thm3.5} we obtain
$$
\|\mathcal{M}^{\alpha}f\|_{q_{\ast},w^{\gamma}}\leq C
\|f\|_{p_{\ast},w^{p_{\ast}\gamma/q_{\ast}}},
$$
for $w\in A_{1}$ and $\gamma,$ where
$$
\gamma=1-\frac{q_{\ast}}{n}=\frac{n-p_{\ast}\alpha-p_{\ast}}{n-p_{\ast}\alpha}.
$$
We have
$$
\gamma\left(\frac{q(x)}{q_{0}}\right)'=\frac{p(x)}{p(x)-p_{\ast}}\cdot\frac{n-\alpha
p_{\ast}}{n}>\frac{n}{n-(1+\alpha)p_{\ast}}\cdot\frac{n-\alpha
p_{\ast}}{n}>1.
$$
It is not hard to proof that
$\gamma\left(\frac{q(\cdot)}{q_{0}}\right)'\in \mathcal{P}_{\log}$
and  $\frac{p_{\ast}q_{\ast}}{q_{\ast}-\gamma
p_{\ast}}=\frac{n}{1+\alpha}.$ From Theorem \ref{thm1.4} we obtain desired
result.
\end{proof}

\

\begin{remark} In case $\alpha=0$ $\mathcal{M}^{0}$ is the well-known
Stein's spherical maximal operator. The
$L^{p(\cdot)}(\mathbb{R}^{n})\rightarrow
L^{p(\cdot)}(\mathbb{R}^{n})$ estimates for maximal operator
$\mathcal{M}^{0}$ was investigate in \cite{FGK}.
\end{remark}

\

\textbf{Bochner-Riesz operators.} The Bochner-Riesz operator in $\mathbb{R}^{n}$, $(n\geq 2)$ are defined for $\beta>0,$ as
$$
\widehat{T_{\beta}^{r}}(\xi)=\left(1-\frac{|\xi|^{2}}{r^{2}}\right)_{+}^{\beta}\widehat{f}(\xi)
$$
with $t_{+}=\max(t,0),$ and the maximal Bochner-Riesz operator is defined by
$$
T_{\beta}^{\ast}f(x)=\sup\limits_{r>0}|T_{\beta}^{r}f(x)|.
$$
\begin{thm} $($\cite{CDL}$)$
If $0<\beta<\frac{n-1}{2},$ then $T_{\beta}^{\ast}$ is bounded on $L^{2}(w^{\frac{2\beta}{n-1}})$ for $w\in A_{2}.$
\end{thm}

\begin{corollary}
 Let $0<\beta<\frac{n-1}{2}$,  $p(\cdot)\in \mathcal{P}_{\log} $ and $\frac{2(n-1)}{n-1+2\beta}<p^{-}\leq p^{+}<\frac{2(n-1)}{n-1-2\beta}.$
Then $T_{\beta}^{\ast}$ is bounded on $L^{p(\cdot)}(\mathbb{R}^{n}).$
\end{corollary}

\end{document}